\documentclass[a4paper,reqno]{amsart}

\makeatletter
\@namedef{subjclassname@2020}{%
  \textup{2020} Mathematics Subject Classification}
\makeatother

\usepackage[toc]{appendix}
\usepackage{amsmath}
\usepackage{amssymb}
\usepackage{amsthm}
\usepackage{mathrsfs}
\usepackage{latexsym}
\usepackage{amscd}
\usepackage{xypic}
\xyoption{curve}
\usepackage{ifthen} 
\usepackage{hyperref} 
\usepackage{graphicx}
\usepackage{enumerate}
\usepackage{enumitem}
\usepackage{rotating}
\usepackage{xcolor}
\usepackage{mathtools}
\usepackage{todonotes}
\usepackage{color,soul}

\usepackage{float}
\restylefloat{table}

\numberwithin{equation}{section}

\def\cocoa{{\hbox{\rm C\kern-.13em o\kern-.07em C\kern-.13em o\kern-.15em A}}}

\newtheorem{theorem}{Theorem}[section]

\newtheorem{proposition}[theorem]{Proposition}
\newtheorem{corollary}[theorem]{Corollary}

\theoremstyle{definition}
\newtheorem{remark}[theorem]{Remark}
\newtheorem{definition}[theorem]{Definition}
\newtheorem{example}[theorem]{Example}

\newcommand {\Aut}{\mathrm{Aut}}
\newcommand {\PGL}{\mathrm{PGL}}

\newcommand {\sHom}{\mathcal{H}\kern -0.25ex{\mathit om}}
\newcommand {\sExt}{\mathcal{E}\kern -0.25ex{\mathit xt}}
\newcommand {\sTor}{\mathcal{T}\kern -0.25ex{\mathit or}}

\newcommand {\rk}{\mathrm{rk}}

\newcommand {\Hom}{\mathrm{Hom}}
\newcommand {\Hilb}{\mathcal{H}\kern -0.25ex{\mathit ilb\/}}

\newcommand {\quantum}{k}

\newcommand {\field}{\mathbf k}

\newcommand{\cC}{{\mathcal C}}

\newcommand{\cE}{{\mathcal E}}

\newcommand{\cN}{{\mathcal N}}
\newcommand{\cO}{{\mathcal O}}
\newcommand{\cG}{{\mathcal G}}
\newcommand{\cT}{{\mathcal T}}
\newcommand{\cI}{{\mathcal I}}

\newcommand {\bZ}{\mathbb{Z}}

\newcommand {\bC}{\mathbb{C}}
\newcommand {\bP}{\mathbb{P}}

\newcommand{\Pic}{\operatorname{Pic}}

\newcommand{\NS}{\operatorname{NS}}

\def\p#1{{\bP^{#1}}}

\title[Tangent, cotangent, normal and conormal bundles]{Tangent, cotangent, normal and conormal\\ bundles are almost never instanton bundles}

\thanks{The  author is a member of GNSAGA group of INdAM}

\subjclass[2020]{Primary: 14J60. Secondary: 14D21, 14F06}

\keywords{Ulrich bundle, Instanton bundle}

\author[G. Casnati]{Gianfranco Casnati}

\begin{document}

\maketitle

\begin{abstract}
In this very short note we give an elementary characteristic free proof of the result claimed in the title (see Theorem \ref{tMain} for a more precise formulation), generalizing a recent result proved in \cite{B--M--PM--T} for Ulrich bundles over the complex field. Moreover, we also give a similar result about the twists of the cotangent bundle and make some comments about the possibility to obtain an analogous result for twists of the tangent bundle.
\end{abstract}

\section{Introduction and Notation}
In this paper a {\sl projective variety $X$} is a closed, integral subscheme of some projective space over an algebraically closed field $\field$ of characteristic $p$. 

In \cite{An--Ca1} the following definition has been introduced.

\begin{definition}
\label{dMalaspinion}
Let $X$ be a projective variety of dimension $n\ge1$ endowed with an ample and globally generated line bundle $\cO_X(h)$.

A non--zero coherent sheaf  $\cE$ on $X$ is called (ordinary) $h$--instanton sheaf with quantum number $\quantum\in\bZ$ if the following properties hold:
\begin{itemize}
\item $h^0\big(\cE(-h)\big)=h^n\big(\cE(-nh)\big)=0$;
\item $h^i\big(\cE(-(i+1)h)\big)=h^{n-i}\big(\cE(-(n-i)h)\big)=0$ if $1\le i\le n-2$;
\item $h^1\big(\cE(-h)\big)=h^{n-1}\big(\cE(-nh)\big)=\quantum$.
\end{itemize}
If $\quantum=0$, then $\cE$ is called $h$--Ulrich sheaf.
\end{definition}

The existence of an instanton sheaf with fixed quantum number $\quantum$ on $X$ is not obvious. E.g. the case $\quantum=0$, i.e. the case of Ulrich sheaves, has been object of deep study in the last two decades and the problem of their existence is still wide open: see \cite{Bea6} and the references therein for more details about this case. 

The interest in dealing with instanton and Ulrich sheaves is also motivated by the fact that their existence on a fixed variety $X$ is often related to interesting geometric properties. E.g. in \cite{An--Ca2} it is shown that when $X\subseteq\p {n+1}$ is a hypersurface and $\cO_X(h)\cong\cO_X\otimes\cO_{\p{n+1}}(1)$, the existence of a locally Cohen--Macaulay $h$--instanton sheaf is equivalent to the existence of a representation of a power of the form defining $X$ as the determinant of a suitable morphism of vector bundles of the same rank on $\p {n+1}$ with a prescribed cohomology table, called Steiner bundles.

Thus, it is perhaps reasonable to ask whether one of the  bundles which are naturally associated to a smooth variety $X$ is an instanton bundle or not and, in the affirmative case, which one is also Ulrich. 

E.g. we can deal with the {\sl cotangent bundle}, i.e. the sheaf of differentials $\Omega_X$, and the {\sl  tangent bundle}, i.e. its dual $\cT_X$: $\Omega_X$ and $\cT_X$ have rank $n:=\dim(X)$. Moreover, if $X\subseteq\p N$ and $\cI_{X}\subseteq\cO_{\p N}$ is its sheaf of ideals, we can also consider two further sheaves, namely the {\sl  conormal bundle}, i.e. $\cC_X:=\cI_{X}/\cI_{X}^2$, and the {\sl normal bundle}, i.e. its dual $\cN_X$: $\cC_X$ and $\cN_X$ have rank $N-n$.

As a preliminary example we consider the bundle $\cN_X$. There are exact sequences
\begin{gather}
\label{seqNormal}
0\longrightarrow\cT_X\longrightarrow\cO_X\otimes\cT_{\p N}\longrightarrow\cN_X\longrightarrow0,\\
\label{seqEuler}
0\longrightarrow\cO_{\p N}\longrightarrow\cO_{\p N}(1)^{\oplus N+1}\longrightarrow\cT_{\p N}\longrightarrow0.
\end{gather}
The restriction to $X$ of  \eqref{seqEuler} combined with \eqref{seqNormal} yields that $\cN_X(-h)$ is certainly globally generated. In particular $h^0(\cN_X(-h))\ne0$: we deduce that $\cN_X$ is never an  $h$--instanton bundle, hence it is never $h$--Ulrich as well. 

In this short note we prove the following result with a very easy and direct characteristic free proof.

\begin{theorem}
\label{tMain}
Let $X\subseteq\p N$ be a smooth projective variety of dimension $n\ge1$ and $\cO_X(h):=\cO_X\otimes\cO_{\p N}(1)$. 

Then the following assertions hold.
\begin{enumerate}
\item $\cT_X$ is an $h$--instanton bundle if and only if either $X\cong\p1$ and $\cO_X(h)\cong\cO_{\p1}(3)$ or $X\cong\p2$ and $\cO_X(h)\cong\cO_{\p2}(2)$: in these cases $\cT_X$ is  $h$--Ulrich.
\item $\Omega_{X}$, $\cN_X$ and $\cC_X$ are never $h$--instanton bundles.
\end{enumerate}
\end{theorem}

As an immediate by--product we obtain the characterization of smooth projective varieties $X\subseteq\p N$ such that $\cT_X$ is Ulrich. Such a characterization has been  proved for the first time with a deep, interesting and long proof when $\field=\bC$  in \cite[Main Theorem]{B--M--PM--T} (see also \cite[Theorem 4.9]{Lop--Ray}).

\begin{corollary}
\label{cMain}
Let $X\subseteq\p N$ be a smooth projective variety of dimension $n\ge1$ and $\cO_X(h):=\cO_X\otimes\cO_{\p N}(1)$. 

Then $\cT_X$ is an $h$--Ulrich bundle if and only if either $X\cong\p1$ and $\cO_X(h)\cong\cO_{\p1}(3)$ or $X\cong\p2$ and $\cO_X(h)\cong\cO_{\p2}(2)$.
\end{corollary}

The next step is to ask whether a twist of the aforementioned bundles is an instanton. 

E.g. if $\field=\bC$, then $\cN_X(-h)$ is an $h$--Ulrich bundle, i.e. an $h$--instanton with quantum number $\quantum=0$, when $X$ is a standard linear determinantal variety (see \cite[Theorem 3.6]{Kl--MR}). 

The complete classification of varieties such that $\cN_X(ah)$ is $h$--Ulrich for some $a\in\bZ$ can be found in \cite[Theorem 1]{Lop}). 
Some very partial results in this direction have been proved without restrictions on $p$ and $\quantum$: see \cite{An--Ca3}. 

When $\field=\bC$ and $\quantum=0$, the behaviour of the twists of the sheaves $\Omega_X$, $\cT_X$, $\cC_X$ is  carefully studied in \cite{Lop--Ray,A--C--L--R}.

We deal with the case $p\ge0$ and $\quantum\ge0$ in Section \ref{sOmega}, proving the following result for the twists of $\Omega_X$.

\begin{theorem}
\label{tOmega}
Let $X\subseteq\p N$ be a smooth projective variety of dimension $n\ge1$ and $\cO_X(h):=\cO_X\otimes\cO_{\p N}(1)$. 

Then  $\Omega_X(ah)$ is an $h$--instanton bundle if and only if $a=2$, $X\cong \p1$ and $\cO_X(h)\cong\cO_{\p1}(1)$.
\end{theorem}

When $n=1$ each $h$--instanton sheaf is $h$--Ulrich by definition, hence the following corollary is an easy consequence of the above theorem. When $\field=\bC$ it is \cite[Proposition 4.1 (i)]{Lop}.

\begin{corollary}
\label{cOmega}
Let $X\subseteq\p N$ be a smooth projective variety of dimension $n\ge1$ and $\cO_X(h):=\cO_X\otimes\cO_{\p N}(1)$. 

Then  $\Omega_X(ah)$ is an $h$--Ulrich bundle if and only if $a=2$, $X\cong \p1$ and $\cO_X(h)\cong\cO_{\p1}(1)$.
\end{corollary}

In Section \ref{sTangent} we list results and examples showing that the problem of determining whether $\cT_X(ah)$ is an $h$--instanton might be more difficult, even when $p=0$: see the recent paper \cite{Lop--Ray}. In particular, we are not able to prove a general result analogous to Theorem \ref{tOmega} above. 

\subsection{Acknowledgments}
The author would like to thank the reviewer for her/his corrections, remarks and suggestions.

\section{Notation and some helpful results}
\label{sGeneral}
Throughout we work over an algebraically closed field $\field$ of arbitrary characteristic $p\ge0$: restrictions on the base field are explicitly indicated when they are assumed. The projective space of dimension $N$ over $\field$ is denoted by $\p N$: $\cO_{\p N}(1)$ is the hyperplane line bundle. The structure sheaf of a scheme $X$ is denoted by $\cO_X$. 

Let $X$ be a smooth projective variety: we set $\omega_X:=\det(\Omega_X)$ and we denote by $K_X$ any divisor such that $\omega_X\cong\cO_X(K_X)$. We recall that  $\Omega_X$ and $\cT_X$ have rank $n:=\dim(X)$, while  $\cC_X$ and $\cN_X$ have rank $N-n$ if $X\subseteq\p N$.

For further notation and all the other necessary results not explicitly mentioned in the paper, we tacitly refer to \cite{Ha2} unless otherwise stated.

In order to prove Theorems \ref{tOmega} and \ref{tMain}, we will make use of some results holding in arbitrary characteristic and concerning the Fujita conjecture on adjoint linear systems (see \cite{Fuj}). To this purpose we recall some definitions and results.

Let $X$ be a smooth variety. A {\sl curve} in $X$ is a closed subscheme of pure dimension $1$. A line bundle $\cO_X(D)$ on $X$ is {\sl nef} if $D\Gamma\ge0$ for each irreducible curve $\Gamma\subseteq X$. Notice that the nefness of $\cO_X(D)$ only depends on its class in the N\'eron--Severi group $\NS(X)$.

A smooth projective variety $X$ of dimension $n\ge2$ is a {\sl scroll on a smooth curve $B$ (with respect to an ample line bundle $\cO_X(\xi)$)} if it is endowed with a surjective morphism $\pi\colon X\to B$ whose fibres are isomorphic to $\p {n-1}$ and such that the restriction of $\cO_X(\xi)$ to them is $\cO_{\p{n-1}}(1)$. 

In this case, there is a rank $n$ vector bundle $\cG$ such that $X$ is the ${\mathbf{Proj}}$ of the symmetric $\cO_B$--algebra of $\cG$. We have $\Pic(X)\cong\bZ\cO_X(\xi)\oplus\pi^*\Pic(B)$ where $\pi_*\cO_X(\xi)\cong\cG$. All the fibres of $\pi$ are algebraically equivalent and we denote by $f\in\NS(X)$ their class. By the Chern equation 
$\xi^n=\deg(\frak g)$ where $\cO_B(\frak g)=\det(\cG)$: by abuse of notation as in \cite{Ha2} we also write
\begin{equation}
\label{ScrollK}
\omega_X\cong\cO_X(-n\xi+(\frak g+K_B)f),
\end{equation}
where we set $\cO_{ X}(\frak a f):=\pi^*\cO_B(\frak a)$ for each divisor $\frak a$ on $B$. Since $\cO_X(\xi)$ is ample, it follows that
\begin{equation}
\label{ScrollDeg}
\deg(\frak g)=\xi^n\ge1.
\end{equation}

\begin{theorem}
\label{tKK}
Let $X$ be a smooth projective variety of dimension $n\ge1$ endowed with an ample line bundle $\cO_X(h)$.

Then either $\omega_X((n-1)h)$ is nef or one of the following assertions holds.
\begin{enumerate}
\item $X\cong\p 2$ and $\cO_X(h)\cong\cO_{\p 2}(2)$.
\item $X\cong\p n$ and $\cO_X(h)\cong\cO_{\p n}(1)$.
\item $X\subseteq\p{n+1}$ is a smooth quadric hypersurface and $\cO_X(h)=\cO_X\otimes\cO_{\p{n+1}}(1)$.
\item $n\ge 2$ and $X$ is a scroll on a smooth curve $B$ with respect to $\cO_X(h)$.
\end{enumerate}
\end{theorem}
\begin{proof}
See \cite[Theorem 1]{Ka--Ko} for the case $n\ge 3$ and when $n\le2$ the comments therein about the validity of the proofs in \cite{Io2, Fuj} in any characteristic.
\end{proof}

The following corollaries are immediate by--products of the above theorem.

\begin{corollary}
\label{cKK}
Let $X$ be a smooth projective variety of dimension $n\ge1$ endowed with an ample line bundle $\cO_X(h)$.

Then either $\omega_X(nh)$ is  nef or $X\cong\p n$ and $\cO_X(h)\cong\cO_{\p n}(1)$.
\end{corollary}
\begin{proof}
The set of varieties $X$ such that $\omega_X(nh)$ is not nef is contained in the one of varieties such that $\omega_X((n-1)h)$ is not nef. Thus it suffices to check that $\p n$ endowed with $\cO_X(h):=\cO_{\p n}(1)$ is the only variety $X$ listed in Theorem \ref{tKK} such that $\omega_X(nh)$ is not nef.  

When $Q\subseteq\p{n+1}$ is a smooth quadric hypersurface, then $\omega_X\cong\cO_X(-nh)$, hence $\omega_X(nh)\cong\cO_X$ is trivially nef.

When $X$ is a scroll on a curve, we use the notation introduced above. Thanks to \eqref{ScrollK} we obtain $\omega_X(nh)\cong\cO_X((\frak g+K_B)f)$ in $\NS(X)$. Thus, it suffices to show
$$
(nh+K_X)\Gamma=(\deg(\frak g)+2p_a(B)-2)f\Gamma\ge0
$$
for each irreducible curve $\Gamma\subseteq X$. 

If $\Gamma$ is contained in a fibre of $\pi$, then $f\Gamma=0$, because the general fibre does not intersect $\Gamma$. If $\Gamma$ is not contained in a fibre of $\pi$, then $\pi_{\vert \Gamma}$ is a finite map of degree $f\Gamma\ge1$. If $p_a(B)\ge1$, then
\begin{equation}
\label{PositiveScroll}
\deg(\frak g)+2p_a(B)-2\ge0
\end{equation}
by \eqref{ScrollDeg}. If $p_a(B)=0$, then $B\cong\p1$, hence $\cG\cong\bigoplus_{i=1}^n\cO_{\p1}(g_i)$. The ampleness of $\cG$ implies $g_i\ge1$, hence \eqref{PositiveScroll} holds in this case as well. We deduce that $\omega_X(nh)$ is nef regardless of $p_a(B)$.
\end{proof}

\begin{corollary}
\label{cSm}
Let $X$ be a smooth projective variety of dimension $n\ge1$ endowed with an ample line bundle $\cO_X(h)$.

Then $\omega_X((n+1)h)$ is  nef.
\end{corollary}
\begin{proof}
We use the same argument of the proof of Corollary \ref{cKK}.
\end{proof}


Assume now that $\cE$ is an $h$--instanton sheaf on a smooth projective variety $X$ of dimension $n$ endowed with an ample and globally generated line bundle $\cO_X(h)$: assume also that either $\cO_X(h)$ is very ample or $p=0$. Thus the following strict restriction holds
\begin{equation}
\label{Slope}
c_1(\cE)h^{n-1}=\frac{\rk(\cE)}2((n+1)h+K_X)h^{n-1},
\end{equation}
see \cite[Theorem 1.6]{An--Ca1}: here $h^n$ and $K_Xh^{n-1}$ denote the degrees of the line bundles $\cO_X(h)$ and $\omega_X$ when $n=1$.

For further notation and all the other results used in the paper we tacitly refer to \cite{Ha2}, unless otherwise stated.

\section{Proof of Theorem \ref{tMain}}
As pointed out in the introduction, the normal bundle is never an instanton bundle. 

We start by listing some easy examples of smooth varieties whose tangent bundle is or is not an instanton bundle: in \cite{B--M--PM--T} the same computations are used solely for dealing with the case $k=0$. 

\begin{example}
\label{eP1}
Let $n=1$. If $X\cong\p1$ and $\cO_X(h)\cong\cO_{\p1}(d)$ with $d\le 3$, then
$$
h^0\big(\cT_X(-h)\big)=h^0\big(\cO_{\p1}(2-d)\big)=3-d.
$$
In all the other cases
$$
h^1\big(\cT_X(-h)\big)=h^0\big(\omega_X^2(h)\big)\ge3p_a(X)-3+\deg(X)\ge1.
$$
Thus $\cT_X$ is an $h$--instanton sheaf, if and only if $X=\p1$ and $\cO_X(h)=\cO_{\p1}(3)$. 

It is immediate to check that $\cT_X$ is the unique rank one $h$--Ulrich  sheaf on $X$.
\end{example}

\begin{example}
\label{eVeronese}
If $X\cong\p2$ and $\cO_X(h)\cong\cO_{\p2}(2)$, then the Bott's formulas imply
\begin{gather*}
h^i\big(\cT_X(-h)\big)=h^{i}\big(\Omega_\p2(1)\big)=0,\qquad h^j\big(\cT_X(-2h)\big)=h^{j}\big(\Omega_\p2(-1)\big)=0,
\end{gather*}
for $i\le 1\le j$. Thus $\cT_X$ is an $h$--instanton sheaf. 

Notice that $\cT_X$ is the unique rank two $h$--Ulrich sheaf on $\p2$: see \cite[Theorem 5.2]{Co--Ge}. 
\end{example}


We are ready to prove Theorem \ref{tMain}.

\begin{proof}[Proof of Theorem \ref{tMain}.]
As pointed out in the introduction, if $X\subseteq\p N$ and $\cO_X(h):=\cO_X\otimes\cO_{\p N}(1)$, then $\cN_X$ is never an $h$--instanton bundle, hence it is not $h$--Ulrich. 

Let us now consider $\cC_X$: in this case \eqref{Slope} becomes
\begin{equation}
\label{Conormal}
(((N-n)(n+1)+2(N+1))h+(N-n+2)K_X)h^{n-1}=0.
\end{equation}
On the one hand, $N>n$ yields
$$
(N-n)(n+1)+2(N+1)=(N-n+2)(n+1)+\lambda
$$
where $\lambda =2(N-n)>0$. Thus \eqref{Conormal} becomes 
$$
(N-n+2)((n+1)h+K_X)h^{n-1}=-\lambda h^n<0,
$$
thanks to the ampleness of $\cO_X(h)$. On the other hand $((n+1)h+K_X)h^{n-1}\ge0$, because $\omega_X((n+1)h)$ is nef by Corollary \ref{cSm}.  The contradiction yields that $\cC_X$ is not an $h$--instanton bundle.

The assertion on $\Omega_X$ is a particular case of Theorem \ref{tOmega}. Anyhow, we can also prove the assertion with the same argument used in the previous case. Indeed, in this case it leads to 
$$
(n-2)((n+1)h+K_X)h^{n-1}=-2(n+1) h^n<0,
$$
again contradicting Corollary \ref{cSm} as in the previous case when $n\ge2$. If $n=1$ the condition $h^1(\Omega_X(-h))=0$, leads to $h^0(\cO_X(h))=0$ by duality, again a contradiction. We deduce that $\Omega_X$ is not an $h$--instanton sheaf.

We now focus our attention on $\cT_X$ in what follows. The case $n=1$ is completely described in Example \ref{eP1}, hence we will assume $n\ge2$ from now on. The equality \eqref{Slope} for $\cT_X$ becomes
\begin{equation}
\label{A}
(n(n+1)h+(n+2)K_X)h^{n-1}=0.
\end{equation}
If $\lambda:=n(n+1)-(n+2)(n-1)>0$, then we have the obvious equality
\begin{equation*}
(n(n+1)h+(n+2)K_X)h^{n-1}=(n+2)((n-1)h+K_X)h^{n-1}+\lambda h^n
\end{equation*}
If $\omega_X((n-1)h)$ is nef, then we can argue as for $\cC_X$, because $\cO_X(h)$ is ample. 

Let us examine the cases listed in Theorem \ref{tKK} when  $\omega_X((n-1)h)$ is not nef. 
In the case (1) of Theorem \ref{tKK} the sheaf $\cT_X$ is an $h$--instanton bundle thanks to Example \ref{eVeronese}. In the case (2) of Theorem \ref{tKK}, we deduce that \eqref{A} becomes $-2=0$, while in case (3) we get $-n=0$: thus $\cT_X$ is not an $h$--instanton sheaf.

Consider the case (4) of Theorem \ref{tKK}. Thus $n\ge2$ and $X\subseteq\p N$ is a scroll on a smooth curve $B$ with respect to $\cO_X(h)$. Thanks to \eqref{ScrollK} and \eqref{ScrollDeg}, then \eqref{A} becomes
$$
\deg(\frak g)=-(2+n)(p_a(B)-1).
$$
It follows that $p_a(B)=0$ necessarily, because the left--hand side is positive and the right--hand one is non--positive if $p_a(B)\ge1$. 

Thus $B\cong\p1$ and $\deg(\frak g)=n+2$, hence $K_{X}=-nh+nf$. If $\cT_{{X}\vert \p1}$ is the relative tangent sheaf of the morphism $\pi\colon {X}\to\p1$,  we have the exact sequence
$$
0\longrightarrow \cT_{{X}\vert \p1}\longrightarrow\cT_{X}\longrightarrow \cO_{X}(2f)\longrightarrow0
$$
because $\pi$ is smooth. Its cohomology tensored by $\cO_\bP(-nh)$ and the Serre duality return
$$
h^n(\cT_{X}(-nh))\ge h^n(\cO_{X}(-nh+2f))=h^0(\cO_{X}((n-2)f))=n-1\ge1.
$$
Thus $\cT_{X}$ is not an $h$--instanton sheaf.
\end{proof}

\begin{remark}
In order to prove Theorem \ref{tMain} we only used that the bundles $\cE$ we are interested in actually satisfy the following properties:
\begin{itemize}
\item $h^0\big(\cE(-h)\big)=0$ (used for $\cN_X$);
\item $h^n\big(\cE(-nh)\big)=0$ (used for curves and occasionally for $\cT_X$);
\item $\cE$ satisfies \eqref{Slope} (used for $\cC_X$, $\Omega_X$ and $\cT_X$)
\end{itemize}
\end{remark}

 \begin{remark}
If the characteristic of $\field$ is zero, then \eqref{Slope} holds only assuming that $\cO_X(h)$ is ample and globally generated, hence the same is true for the assertions about $\Omega_X$ and $\cT_X$ in Theorem \ref{tMain}.
\end{remark}

\begin{remark}
If $X\subseteq\p N$, then the same argument used in the proof of Theorem \ref{tMain} easily implies that $\cO_X\otimes\Omega_{\p N}$ and $\cO_X\otimes\cT_{\p N}$ are never instanton bundles with respect to $\cO_X(h):=\cO_X\otimes\cO_{\p N}(1)$.
\end{remark}

\section{Proof of Theorem \ref{tOmega}.}
\label{sOmega}
We already checked in Theorem \ref{tMain} that $\Omega_X$ is never an $h$--instanton bundle with a very short proof. In this section we  deal with the twists of the cotangent bundle, giving  the proof of Theorem \ref{tOmega} stated in the introduction.

We start with the following example analyzing the case $n=1$: in \cite{Lop} essentially the same computations are used solely for dealing with the case $k=0$. 

\begin{example}
\label{eOmega}
Assume $n=1$. On the one hand, if $\Omega_X(ah)$ is an $h$--instanton, then 
$$
h^0(\Omega_X((a-1)h))=h^1(\Omega_X((a-1)h))=0:
$$
in particular 
$$
h^0(\cO_X((1-a)h))=h^1(\Omega_X((a-1)h))=0,
$$
hence $a\ge2$, because $\cO_X(h)$ is globally generated. On the other hand, the Riemann--Roch theorem on $X$ implies
$$
h^0(\Omega_X((a-1)h))-h^1(\Omega_X((a-1)h))=p_a(X)-1+(a-1)\deg(X).
$$
Since $a\ge2$, it follows that $p_a(X)=0$ necessarily and, consequently, $a=2$ and $\deg(X)=1$, i.e. $X\cong\p1$ and $\cO_X(h)\cong\cO_{\p1}(1)$. 

Conversely, it is immediate to check that $\Omega_{\p1}(2)\cong\cO_{\p1}$ is the unique rank one instanton (and Ulrich) sheaf on $\p1$ with respect to $\cO_{\p1}(1)$.
\end{example}

We now prove Theorem \ref{tOmega} stated in the introduction.

\begin{proof}[Proof of Theorem \ref{tOmega}.]
If $\Omega_X(ah)$ is an $h$--instanton bundle, then \eqref{Slope} implies
\begin{equation}
\label{Omega}
(n^2+(1-2a)n)h^n+(n-2)K_Xh^{n-1}=0.
\end{equation}

If $n=1$, then  $a=2$, $X\cong\p 1$ and $\cO_X(h)\cong\cO_{\p1}(1)$ by Example \ref{eOmega}. If $n=2$, then \eqref{Omega} has no integral solutions, hence $\Omega_X(ah)$ is not an $h$--instanton. Thus, the proof is complete also when $n=2$.

Assume $n\ge3$. We have
$$
n^2+(1-2a)n= n(n-2)+n(3-2a).
$$
Let $a\le 1$. On the one hand, $\lambda=n(3-2a)>0$, hence \eqref{Omega} becomes
\begin{equation}
\label{Referee}
(n-2)(nh^n+K_Xh^{n-1})=-\lambda h^n<0,
\end{equation}
due to the ampleness of $\cO_X(h)$. On the other hand the left--hand side of the equality above is non--negative by Corollary \ref{cKK} unless $X\cong\p n$ and $\cO_X(h)\cong\cO_{\p n}(1)$.

If this is the case, then \eqref{Omega} returns $n(1-a)=-1$ which is impossible as $n\ge3$. Thus the left--hand side of \eqref{Referee} is non negative, while the right--hand side is negative, a contradiction. It follows that $a\ge2$. 

Notice that the exterior product of the dual of \eqref{seqEuler} tensored by $\cO_{\p N}(a)$ is
$$
0\longrightarrow(\wedge^2\Omega_{\p N})(a)\longrightarrow\cO_{\p N}(a-2)^{\oplus{{N+1}\choose2}}\longrightarrow\Omega_{\p N}(a)\longrightarrow0.
$$
Its restriction to $X$ combined with the surjective morphism $\cO_X\otimes\Omega_{\p N}(a)\twoheadrightarrow\Omega_X(ah)$ induced by the dual of \eqref{seqNormal} implies that $\Omega_X(ah)$ is globally generated for $a\ge2$. Since we must have $h^0(\Omega_X((a-1)h))=0$ by definition, it follows that $a\le2$ necessarily. Thus, if $n\ge3$ and $\Omega_X(ah)$ is an $h$--instanton, then $a=2$ necessarily.

By definition $\Omega_X(2h)$ is not an $h$-instanton bundle when $n\ge3$ if 
\begin{equation}
\label{Hodge}
h^1(\Omega_X)\ge1.
\end{equation}

If $\Omega_X(2h)$ is an $h$--instanton bundle, then $h^0(\Omega_X(h))=0$ by definition. Thus, tensoring \eqref{seqEuler} by $\Omega_X$ we obtain an injective map
$$
\varrho\colon \Hom_X(\cT_X,\cO_X\otimes\cT_{\p N})\cong H^0(\Omega_X\otimes\cT_{\p N})\longrightarrow H^1(\Omega_X).
$$
Thus \eqref{seqNormal} implies $\varrho\ne0$, which yields \eqref{Hodge}. Thus, $\Omega_X(2h)$ is not an $h$-instanton bundle when $n\ge3$.
\end{proof}

\begin{remark}
\label{rChar}
If $\field=\bC$, then \eqref{Hodge} certainly holds because the Lefschetz  $(1,1)$--theorem implies the existence of an injective morphism $\NS(X)\to H^1(\Omega_X)$, hence there would be no need of further computations in this case. 

When $p\ne0$ the above morphism still exists, but it could be not injective, hence we cannot argue \eqref{Hodge} in the same way. 
\end{remark}

\section{On the tangent bundle of some canonical varieties}
\label{sTangent}
In this section we collect some partial results and examples showing that the problem of determining whether $\cT_X(ah)$ is an $h$--instanton bundle might be highly non--trivial (hence, perhaps, quite intriguing).

The following result is an immediate consequence of Theorem \ref{tMain}.

\begin{proposition}
\label{pSubcanonical}
Let $X\subseteq\p N$ be a smooth projective variety of dimension $n\ge1$ and $\cO_X(h):=\cO_X\otimes\cO_{\p N}(1)$. Assume that $\omega_X\cong\cO_X(\alpha h)$ for some $\alpha\in\bZ$.

Then  $\cT_X(ah)$ is an $h$--instanton bundle if and only if $a=\alpha +n-1$, $X\cong \p1$ and $\cO_X(h)\cong\cO_{\p1}(1)$;
\end{proposition}
\begin{proof}
Since $\omega_X\cong\cO_X(\alpha  h)$, it follows from the Serre duality that $\cT_X(ah)$ is an $h$--instanton bundle if and only if the same is true for $\Omega_X((\alpha -a+n+1)h)$ (see \cite[Section 6]{An--Ca1}). Thus the statement follows easily from Theorem \ref{tOmega}.
\end{proof}

In view of the above proposition, it is perhaps natural to deal with the pluricanonical and antipluricanonical varieties $X$, i.e. such that $\omega_X^{\beta}\cong\cO_X(h)$ for some $\beta\in\bZ$. Trivially $\beta\ne0$ and the case $\beta=\pm1$ is covered by Proposition \ref{pSubcanonical}, hence we assume $\beta\not\in\{\ 0,\pm1\ \}$ in the following statement.

\begin{proposition}
\label{pPluriCanonical}
Let $X$ be a smooth projective variety of dimension $n\ge1$ endowed with an ample and globally generated line bundle $\cO_X(h)$. Assume that $\omega_X^{\beta}\cong\cO_X(h)$ for some $\beta\in\bZ\setminus\{\ 0,\pm1\ \}$.

If $\cT_X(ah)$ is an $h$--instanton bundle, then $n=2$, $1\le a\le 2$ and $1\le K_X^2\le 5\chi(\cO_X)$. In this case the quantum number of $\cT_X(ah)$ is $10\chi(\cO_X)-2K_X^2$.
\end{proposition}
\begin{proof}
Assume that $\cT_X(ah)$ is an $h$--instanton bundle. The hypothesis on $\omega_X$ and  \eqref{Slope} yield
$n(2a-n-1)\beta=n+2$, because $0\ne h^n=\beta^n K_X^n$.
In particular $2a\ne n+1$ because $n\ge1$. The above equality has no integral solution if $n=1$, hence we will assume $n\ge2$ from now on: thus $1<(n+2)/n\le 2$. Since 
$$
\beta=\frac{n+2}{n(2a-n-1)}\in\bZ,
$$
it follows that necessarily $(n+2)/n=2$ and $2a-n-1\in\{\ \pm1,\pm2\ \}$. We deduce that $n=2$, hence $\beta=\pm2$ (recall that $\beta\ne\pm1$): consequently $2a=3\pm1$, whence $1\le a\le 2$.

If $\cT_X(ah)$ is an $h$--instanton, then $h^0(\cT_{ X}((a-1)h))=h^2(\cT_{ X}((a-2)h))=0$  by definition: the Serre duality then implies $h^2(\cT_{ X}((a-1)h))=0$ as well. We have
\begin{gather*}
c_1(\cT_X((a-1)h))=2(a-1)h-K_X,\\
c_2(\cT_X((a-1)h))=12\chi(\cO_X)-K_X^2-(a-1)hK_X+(a-1)^2h^2,
\end{gather*}
hence the Riemann--Roch theorem returns
$$
\quantum=h^1(\cT_X((a-1)h))=-\chi(\cT_X((a-1)h))=10\chi(\cO_X)-K_X^2-((a-1)h-K_X)^2.
$$
Thus, for $1\le a\le 2$, we obtain $\quantum=10\chi(\cO_X)-2K_X^2\ge0$ whence $K_X^2\le 5\chi(\cO_X)$. On the other hand $\beta^2 K_X^2=h^2\ge1$, whence $K_X^2\ge1$.
\end{proof}

In the following examples, we inspect the surfaces in Proposition \ref{pPluriCanonical} in more detail, also showing that $q(X)=0$ necessarily. For simplicity we assume $p=0$.

Notice that the existence on such surfaces of rank two $h$--instanton bundles $\cE$ with $c_1(\cE)=3h+K_X=(3\beta+1)K_X$ and arbitrary quantum number follows from \cite[Example 6.11]{An--Ca1}, because $h^1(\omega_X^{\beta})=0$ thanks to the Kodaira vanishing theorem, because $\beta\in\bZ\setminus\{\ 0,\pm1\ \}$.

\begin{example}
\label{eDelPezzo}
Let $a=1$, hence $\beta=-2$. 

In this case $\cO_X(h)\cong\omega_X^{-2}$: in particular $\omega_X^{-1}$ is ample, hence $X$ is a Del Pezzo surface. Every Del Pezzo surface is either the blow up of $\p2$ at $0\le r\le 8$ general points or it is isomorphic to $\p1\times\p1$. In the former case $K_X^2=9-r$ is called {\sl degree} of $X$ and $\omega_X^{-1}$ is globally generated if $r\le 7$ and very ample if $r\le6$. In the latter case $\omega_X^{-1}$ is very ample and $K_X^2=8$.

The bundle $\cT_X(h)$ is an $h$--instanton if and only if $h^0(\cT_X)=0$, because it is orientable, i.e. $c_1(\cT_X(h))=3h+K_X=-5K_X$ (see \cite[Corollary 6.9]{An--Ca1}): moreover, in this case, its quantum number is  $h^1(\cT_X)=h^1(\cT_X(-h))$. We recall that $h^0(\cT_X)$ is the dimension of the tangent space to $\Aut(X)$ at the identity because $p=0$ (see \cite[Exercise I.2.16.4]{Kol}): it follows that $\cT_X(h)$ is an $h$--instanton if and only if $\Aut(X)$ is finite. If $2\le K_X^2\le5$, then $\Aut(X)$ is finite (see \cite[Corollary 8.2.33]{Dol}. If $K_X^2\ge6$, then either $X$ is $\p2$ blown up at $0\le r\le 3$ general points, hence $\Aut(X)$ has positive dimension because it contains as subgroup the group of projectivities of $\p2$ fixing the blown up points, or $X\cong\p1\times\p1$, hence it contains $\PGL_2\times\PGL_2$ as subgroup.

We conclude that $\cT_X(h)$ is an $h$--instanton if and only if $2\le K_X^2\le5$ and it is $h$--Ulrich if and only if $K_X^2=5$.

This result is a very particular case of a  more general result proved in \cite{Lop--Ray}.
\end{example}


\begin{example}
\label{eGeneralType}
Let $a=2$, hence $\beta=2$. 

In this case $\cO_X(h)\cong\omega_X^{2}$: in particular $X$ is a surface of general type and it is minimal because $\omega_X$ is ample.
Moreover, $\cT_X(2h)$ is an $h$--instanton  if and only if $h^0(\cT_{ X}(h))=h^0(\cT_{ X}(2K_X))=0$. If $p_g(X)\ge1$, we then obtain
$$
0=h^0(\cT_X(2K_X))=h^0(\Omega_X^1(K_X))\ge h^0(\Omega_X^1)=h^1(\cO_X)=q(X),
$$
because $\cT_X(K_X)\cong\Omega_X^1$. If $p_g(X)=0$, then $q(X)=0$, because $\chi(\cO_X)\ge1$ (see \cite[Theorem VII.1.1 (ii)]{B--H--P--VV}).

Thus $X$ is embedded by $\cO_X(h)$ as a surface of degree $d:=h^2=4K_X^2$ inside $\p N$ where $N=h^0(\omega_X^{2})-1=K_X^2+p_g(X)$ (see \cite[Corollary VII.5.4]{B--H--P--VV}). 

If $N=K_X^2+p_g(X)=3$, then $\omega_X\cong\cO_X((d-4) h)$ thanks to the adjunction formula in $\p3$. If $N=K_X^2+p_g(X)=4$, then the double point formula (see \cite[Example A.4.1.3]{Ha2}) implies
$$
d^2-13d+12\chi(\cO_X)=0.
$$
because $d=4K_X^2$. Since $\chi(\cO_X)=1+p_g(X)\ge1$, it follows that the only possible cases for $K_X^2$ being a positive integer such that $d=4K_X^2$ is a solution of the equation above are either $p_g(X)=0$ and $K_X^2=3$ or $p_g(X)=2$ and $K_X^2=1$. 
In both cases $K_X^2+p_g(X)=3$, contradicting the hypothesis $N=4$, hence $N=K_X^2+p_g(X)\ge5$. In particular $K_X^2\ge6-\chi(\cO_X)$.

Moreover, the classification of surfaces of degree up to $8$ in $\p N$ (see \cite{Io1,Io3}: see also \cite{Ok3}) implies that there is no surface $X$ with $\kappa(X)=2$ and $K_X^2\le 2$. Thus $K_X^2\ge3$ necessarily. 

We do not know if a surface $X$ such that $h^0(\cT_X(h))=0$ actually exists, but if it does, $X$ is a minimal surface of general type such that
\begin{gather*}
\cO_X(h)\cong\omega_X^{2},\qquad q(X)=0,\qquad \max\{\ 3,6-\chi(\cO_X)\ \}\le K_X^2\le 5\chi(\cO_X),
\end{gather*}
and $\cT_X(2h)$ is $h$--Ulrich if and only if $K_X^2=5\chi(\cO_X)$. In particular, either $\cT_X(2h)$ is $h$--Ulrich or  $h^0(\cT_{ X}(h))=h^0(\cT_{ X}(2K_X))\ne0$ when $p_g(X)=0$.
\end{example}

\begin{remark}
\label{rSharp}
We could also look for varieties $X$ with $\omega_X^{\beta}\cong\cO_X(\alpha h)$ and such that $\cT_X(ah)$ is an $h$--instanton bundle, for suitable $\alpha,\beta, a\in\bZ$. 

E.g. the Veronese surface satisfies the above hypothesis with $\alpha=3$, $\beta=-2$, $a=0$ (see Example \ref{eVeronese} above or \cite{B--M--PM--T}): notice that in this case $\cO_X(h)\cong\cO_{\p2}(2)$ and $\cT_X$ is actually $h$--Ulrich.

We refer the interested reader to \cite{Lop--Ray} for further results, examples and details.
\end{remark}

\bigskip
\noindent
Gianfranco Casnati,\\
Dipartimento di Scienze Matematiche, Politecnico di Torino,\\
c.so Duca degli Abruzzi 24,\\
10129 Torino, Italy\\
e-mail: {\tt gianfranco.casnati@polito.it}

\end{document}